\documentclass[11pt]{article}

\usepackage{hyperref}
\usepackage[draft]{fixme}
\usepackage{mathrsfs}
\usepackage[a4paper]{geometry}
\usepackage{amssymb, amsthm, amsfonts, amsxtra, amsmath}
\usepackage{latexsym}
\usepackage{mathabx}
\usepackage{pdfpages}
\usepackage{epic}
\usepackage{fouridx}
\usepackage{parskip} % paragraphs have no indents and vertical spacings inbetween
\makeatletter % need this to avoid the conflict between amsthm and parskip
\def\thm@space@setup{%
  \thm@preskip=\parskip \thm@postskip=0pt
}
\makeatother
\usepackage{enumerate}
\usepackage{bm}

\newtheorem{Theorem}{Theorem}[section]
\newtheorem*{Theo*}{Theorem}
\newtheorem{Lem}[Theorem]{Lemma}
\newtheorem{Prop}[Theorem]{Proposition}
\newtheorem{Cor}[Theorem]{Corollary}

\theoremstyle{definition}
\newtheorem{Def}[Theorem]{Definition}
\newtheorem{Rem}[Theorem]{Remark}

\DeclareMathOperator{\rd}{\mathrm{d}}
\DeclareMathOperator{\graph}{\mathrm{Graph}}
\DeclareMathOperator{\mult}{\mathrm{mult}}
\DeclareMathOperator{\C}{\mathbb{C}}
\DeclareMathOperator{\R}{\mathbb{R}}
\DeclareMathOperator{\ssl}{\mathfrak{s}\mathfrak{l}}
\DeclareMathOperator{\su}{\mathfrak{s}\mathfrak{u}}

\DeclareMathOperator{\g}{\mathfrak{g}}
\DeclareMathOperator{\Alt}{\mathrm{Alt}}
\DeclareMathOperator{\id}{\mathrm{id}}
\DeclareMathOperator{\ee}{\mathfrak{e}}
\DeclareMathOperator{\bb}{\mathfrak{b}}
\DeclareMathOperator{\Bb}{\mathfrak{B}}
\DeclareMathOperator{\ext}{\mathrm{ext}}
\DeclareMathOperator{\Ree}{\mathrm{Re}}
\DeclareMathOperator{\Imm}{\mathrm{Im}}
\DeclareMathOperator{\Arg}{\mathrm{Arg}}

\date{}

\numberwithin{equation}{section}

\begin{document}
\title{Poisson-Lie groupoids and the contraction procedure}
\author{Kenny De Commer\thanks{Department of Mathematics, Vrije Universiteit Brussel, VUB, B-1050 Brussels, Belgium, email: {\tt kenny.de.commer@vub.ac.be}. This work was partially supported by NCN grant 2012/06/M/ST1/00169 and FWO grant G.0251.15N.} }
\maketitle

\begin{abstract}
\noindent On the level of Lie algebras, the contraction procedure is a method to create a new Lie algebra from a given Lie algebra by rescaling generators and letting the scaling parameter tend to zero. One of the most well-known examples is the contraction from $\su(2)$ to $\ee(2)$, the Lie algebra of upper-triangular matrices with zero trace and purely imaginary diagonal. In this paper, we will consider an extension of this contraction by taking also into consideration the natural bialgebra structures on these Lie algebras. This will give a bundle of central extensions of the above Lie algebras with a Lie bialgebroid structure having transversal component. We consider as well the dual Lie bialgebroid, which is in a sense easier to understand, and whose integration can be explicitly presented. 
\end{abstract}

\section*{Introduction}

The \emph{contraction procedure} allows one to see a fixed Lie group as living in a bundle of Lie groups having a non-trivial, degenerate limit at some boundary point. The most well-known example is the contraction of $SU(2)$ to $E(2)$, the group of upper-triangular matrices with determinant one and unimodular diagonal, by means of the intermediary groups \[SU(2)_y = \{\begin{pmatrix} a & b \\ -y\bar{b} & \bar{a}\end{pmatrix} \mid |a|^2+y|b|^2 = 1\}\subseteq SL(2,\C).\] See for example \cite[Chapter 10]{Gil94} for a general discussion on the contraction procedure. 

The contraction method also works well on the level of quantum groups, see \cite{Wor92} for the contraction of quantum $SU(2)$ to quantum $E(2)$. However, something remarkable happens in the quantum case: instead of a limit, one can also take a direct `jump' from $SU(2)$ to $E(2)$ by means of the \emph{reflection procedure} \cite{DeC12}. On the level of the associated quantized universal enveloping algebras, this is explained by the fact that not only can one deform $U_q(\su(2))$ to $U_q(\ee(2))$ by a limit of Hopf algebras $U_q(\su(2))_y$, but this family can even be extended to a two-parameter family of algebras $U_q(\su(2))_{x,y}$ which now no longer have an internal comultiplication, but a compatible three-parameter family of `\emph{external}' comultiplications turning the family into a \emph{cogroupoid} \cite{BiC14}. Informally, the dual of $U_q(\su(2))_{0,y}$ is then the object allowing one to `jump' from $SU_q(2)$ to $E_q(2)$. 

The goal of this article is to describe a semi-classical analogue of this observation. Namely, motivated by the fact that $U_q(\su(2))$ is a deformation of the function algebra on the Poisson-Lie dual $B$ of $SU(2)$, we show that $B$ can be endowed with a two-parameter family of Poisson structures which are \emph{Poisson bitorsors}. Globally, this turns the trivial product groupoid $\R\times B\times \R$ into a \emph{Poisson groupoid}. The associated infinitesimal object is then a Lie bialgebroid over $\R$. Taking the dual Lie bialgebroid then gives, rather surprisingly, a bundle of \emph{four-dimensional} Lie algebras forming a central extension of the bundle contracting $\su(2)$ to $\ee(2)$. Moreover, the integrated Poisson groupoid is then a bundle of 4-dimensional Lie groups but with a transversal Poisson structure. The difficulty to give a full description of this integrated Poisson groupoid reflects in a way the subtle analytic behaviour appearing in the quantum context \cite{DeC12}.

The article is structured as follows. In a \emph{first section}, we give some preliminaries on the contraction procedure and on Lie bialgebras. In the \emph{second section}, we describe the Lie bialgebroid associated to the two-parameter family of Poisson structures on $B$, as well as the dual Lie bialgebroid which, as an algebroid, is bundle of 4-dimensional Lie algebras. In the \emph{third section}, we then consider the integrated Poisson groupoids of these Lie bialgebroids.

\section{Dual pairs of Lie bialgebras}

Consider the real Lie algebra $\su(2)$, generated by three elements $J_1,J_2,J_3$ such that, taking indices $\mathrm{mod}$ 3, \[\lbrack J_i, J_{i+1}\rbrack  = 2J_{i+2}.\] When considering matrix realisations, we will use the identification \[J_1 = \begin{pmatrix} i &0 \\ 0 & -i\end{pmatrix},\quad J_2 = \begin{pmatrix} 0 & 1 \\ -1 & 0\end{pmatrix},\quad J_3 = \begin{pmatrix} 0 & i \\ i & 0\end{pmatrix}.\]

Recall now the notion of a \emph{Lie bialgebra}, for which we will take \cite{ES02} as our main reference. We will use here the natural extension of the Lie bracket on $\g$ to a Gerstenhaber algebra structure on $\Lambda \g$, so that for example \[\lbrack a,b\wedge c\rbrack = \lbrack a,b\rbrack \wedge c +b\wedge \lbrack a,c\rbrack.\] We will also identify $\Lambda^n(\g) \subseteq \g^{\otimes n}$ by the map \[a_1\wedge\ldots \wedge a_n\mapsto \sum_{\sigma\in S_n} (-1)^{\sigma}a_{\sigma(1)}\otimes \ldots \otimes a_{\sigma(n)}.\] 

\begin{Def} A \emph{Lie bialgebra} consists of a Lie algebra $\g$ together with a linear map \[\delta: \g\rightarrow \Lambda^2 \g,\] called the \emph{cobracket}, satisfying the cocycle identity \[\delta(\lbrack a,b\rbrack) = \lbrack \delta(a),b\rbrack + \lbrack a,\delta(b)\rbrack,\qquad \forall a,b\in \g,\] and the \emph{coJacobi identity} \[\Alt (\delta\otimes \id)\delta(a) = 0,\qquad \forall a\in \g,\] where $\Alt(a_0\otimes a_1\otimes a_2) = \sum_{k=0}^2 a_{k}\otimes a_{k+1}\otimes a_{k+2}$, the indices treated cyclically.
\end{Def} 

For example, it is easy to check that $\su(2)$ becomes a bialgebra by the cobracket \[\delta(J_1) =0, \qquad \delta(J_2) = J_1\wedge J_2,\quad \delta(J_3) = J_1 \wedge J_3. \] which we will call the \emph{standard} Lie bialgebra structure. 

A finite-dimensional Lie bialgebra $(\g,\delta)$ has a canonical dual Lie bialgebra, obtained by endowing $\g^*$ with the bracket \[\lbrack \omega,\theta\rbrack (a) = (\omega\otimes \theta)\delta(a),\] and the cobracket \[\delta_*(\omega)(a\otimes b) = \omega(\lbrack a,b\rbrack).\] 

For $\g = \su(2)$, let $\{\hat{J}_i\}$ be the canonical dual basis to $\{J_i\}$. Then the Lie bialgebra structure on $\su(2)^*$ is easily computed to be determined by  \[ \lbrack \hat{J}_1,\hat{J}_{2}\rbrack =  \hat{J}_2,\quad  \lbrack \hat{J}_1,\hat{J}_{3}\rbrack =  \hat{J}_3,\quad \lbrack \hat{J}_2,\hat{J}_3\rbrack =0,\] with dual cobracket \[\hat{\delta}(\hat{J}_i) = 2 \hat{J}_{i+1}\wedge \hat{J}_{i+2}.\] This means that $\su(2)^*$ is isomorphic to $\bb = \R\ltimes \C$, which can be identified with lowertriangular complex matrices with zero trace and real diagonal, by means of the identification \begin{equation}\label{EqMatDual}2\hat{J}_1 = \begin{pmatrix} -1 & 0 \\ 0 & 1\end{pmatrix},\quad 2\hat{J}_2  = \begin{pmatrix} 0 & 0 \\ i & 0\end{pmatrix},\quad 2\hat{J}_3 = \begin{pmatrix} 0 & 0 \\ 1 & 0 \end{pmatrix}.\end{equation}

We can deform the above situation with an extra parameter $y\in \R$. Namely, consider $\su(2,y)$, which is the Lie bialgebra defined in exactly the same way as $\su(2)$, \emph{except} that, when writing the generators as $J_i^{(y)}$, the commutation relation between $J_2^{(y)}$ and $J_3^{(y)}$ now becomes \[\lbrack J_2^{(y)},J_3^{(y)}\rbrack  = 2y J_1^{(y)}.\] This means that for $y>0$, we get an isomorphic bialgebra copy of $\su(2)$, while for $y<0$, we get an isomorphic copy of $\su(1,1)$. For $y=0$, we rather get the Lie bialgebra $\ee(2)$ whose underlying real Lie algebra is $\cong i\R\ltimes \C$, and which can be realized as uppertriangular complex matrices with zero trace and purely imaginary diagonal. In general, these deformed real Lie algebras can be realized as Lie subalgebras of $\ssl(2,\C)$ by the identifications \[J_1^{(y)} = \begin{pmatrix} i &0 \\ 0 & -i\end{pmatrix},\quad J_2^{(y)} = \begin{pmatrix} 0 & 1 \\ -y & 0\end{pmatrix},\quad J_3^{(y)} = \begin{pmatrix} 0 & i \\ iy & 0\end{pmatrix}.\]

In this case, performing the dual construction gives for each $y$ the same Lie algebra $\bb$, but the cobracket now has \[\hat{\delta}(\hat{J_1}^{(y)}) = 2y \hat{J}_2^{(y)}\wedge \hat{J}_3^{(y)}.\]

One can see the above one-parameter family of Lie bialgebras as a field of Lie bialgebras, which `contract' as $y\rightarrow 0$, see \cite{BdO97}. In the following section, we will go one step further and extend this field to a bigger \emph{Lie bialgebroid}.

\section{A contraction Lie bialgebroid}

Our main reference for Lie (bi)algebroids and Poisson and Lie groupoids will be \cite{L-GSX11}. 

We first recall the notion of Lie algebroid.

\begin{Def}\label{DefLieAlgbr} A \emph{Lie algebroid} over a smooth manifold $\mathscr{O}$ consists of a vector bundle $A$ over $\mathscr{O}$ together with a vector bundle map $\rho: A\rightarrow T\mathscr{O}$, called the \emph{anchor map}, and a Lie algebra structure $\lbrack\,\cdot\,,\cdot\,\rbrack$ on $\Gamma(A)$, the space of sections of $A$, s.t.~ $\lbrack X,\cdot \,\rbrack_A$ satisfy Leibniz rule w.r.t.~ $\rho(X)$ for each $X \in \Gamma(A)$,\[\lbrack X, fY\rbrack  = f\lbrack X,Y\rbrack + \rho(X)(f)Y,\quad Y \in \Gamma(A),f\in C^{\infty}(\mathscr{O}).\]
\end{Def}

For example, any Lie algebra is a Lie algebroid over the one-point set. Also, the tangent vector bundle of a manifold becomes a Lie algebroid in a natural way. Finally, one can take direct sums of Lie algebroids. 

In the following, we will be interested in the direct sum Lie algebroid  \[\Bb_{\R} = \bb\oplus \, T\R\] over $\R$, with anchor map $\widehat{\rho}:(a,t)\mapsto t$. However, we want to deform the cobracket `in the direction of $\R$'. For this, we will need the following notion of \emph{Lie bialgebroid}.

\begin{Def}[\cite{L-GSX11}, Remark 3.14] A \emph{Lie bialgebroid} consists of a Lie algebroid $(A,\mathscr{O},\rho)$ such that also the dual bundle $(A^*,\mathscr{O})$ carries a Lie algebroid structure whose anchor map we denote $\rho_*$, and such that, writing $\langle\,\cdot\,,\cdot\,\rangle$ for the $C^{\infty}(\mathscr{O})$-valued pairing between $\Gamma(\Lambda A)$ and $\Gamma(\Lambda A^*)$, \begin{equation}\label{Eqdeltadual} \delta: \Gamma(A)\rightarrow \Gamma(\Lambda^{2}(A)),\quad \langle \delta(X),\xi\wedge\eta\rangle = \frac{1}{2}\left(\langle X,\lbrack \xi,\eta\rbrack\rangle -\rho_*(\xi)(\langle X,\eta\rangle) + \rho_*(\eta)(\langle X,\xi\rangle)\right)\end{equation} satisfies the cocycle condition \[\delta(\lbrack X,Y\rbrack) = \lbrack \delta(X),Y\rbrack + \lbrack X,\delta(Y)\rbrack.\]
\end{Def}
 
In the last equation, we use again the natural extension of $\lbrack\,\cdot\,,\cdot\,\rbrack$ to a Gerstenhaber algebra structure on $\Gamma(\Lambda A)$. Note that, by definition, \[\delta (fX) = f\delta(X)+\delta(f)\wedge X,\] where $\delta(f) = \rho^*(df)$ for $\rho^*$ dual to $\rho_*$. In fact, $\delta$ can be extended to a degree one derivation of $\Gamma(\Lambda A)$ (for its ordinary exterior algebra structure). There is a natural duality for Lie bialgebroids, simply by interchanging $(A,\mathscr{O},\rho)$ and $(A^*,\mathscr{O},\rho_*)$. 

To put a Lie bialgebroid structure on $\Bb_{\R}$, we will introduce directly the dual Lie algebroid structure. In fact, the dual will simply be a field of \emph{four-dimensional} Lie algebras, which we introduce in the next lemma. 

\begin{Lem} For $y\in \R$, let $\su_{\ext}^{(y)}(2)$ be a four-dimensional vector space generated by basis elements $\{J_0^{(y)},J_1^{(y)},J_2^{(y)},J_3^{(y)}\}$. We define on $\su_{\ext}^{(y)}(2)$ a bracket $\lbrack\,\cdot\,,\cdot\,\rbrack$ defined as \[ \lbrack J_1^{(y)},J_2^{(y)}\rbrack =2J_3^{(y)},\quad \lbrack J_2^{(y)},J_3^{(y)}\rbrack = 2yJ_1^{(y)}+2J_0,\quad \lbrack J_3^{(y)},J_1^{(y)}\rbrack =2J_2^{(y)},\] and  $J_0$ central. Then $\su_{\ext}^{(y)}(2)$ is a four-dimensional Lie algebra.
\end{Lem} 

\begin{proof} It is trivial to check that, for $y\neq 0$, we have an isomorphism of $\su_{\ext}^{(y)}(2)$ with the direct sum Lie algebra $\su(2,y)\oplus \R$. Hence $\su_{\ext}^{(y)}(2)$ is a Lie algebra. By continuity, also $\su_{\ext}^{(0)}(2)$ is a Lie algebra.
\end{proof} 

Let now $\mathfrak{A}_{\R}$ be the trivial vector bundle over $\R$ with fiber $\su_{\ext}^{(y)}(2)$ at $y$. Then $\mathfrak{A}_{\R}$ forms a bundle of Lie algebras, and hence a Lie algebroid with respect to the trivial anchor map. Moreover, we can naturally pair $\mathfrak{A}_{\R}$ with $\Bb_{\R}$ by making the constant sections \[J_i: y \mapsto J_i^{(y)}\] dual to the following sections of $\Bb_{\R}$, \[\hat{J}_i = (\hat{J}_i,0) \textrm{ for }i\in \{1,2,3\}, \quad \hat{J}_0 = (0,2\frac{\partial}{\partial y}).\]

\begin{Prop} The above pairing of Lie algebroids make $\mathfrak{A}_{\R}$ and $\Bb_{\R}$ into (dual) Lie bialgebroids.
\end{Prop} 

\begin{proof} 
Write $Y\in C^{\infty}(\R)$ for the identity function $y\mapsto y$. Consider the following elements in $\Gamma(\Lambda^2\Bb_{\R})$:  \begin{equation}\label{Eqdelta0} \widehat{\delta}(\hat{J}_2) = 2\hat{J}_3\wedge \hat{J}_{1},\quad  \widehat{\delta}(\hat{J}_3) = 2\hat{J}_1\wedge \hat{J}_2,\quad \widehat{\delta}(\hat{J}_1) = 2Y\hat{J}_2\wedge \hat{J}_3,\quad \widehat{\delta}(\hat{J}_0) = 2Y\hat{J}_2\wedge \hat{J}_3.\end{equation}
We can extend this to a linear map \[\widehat{\delta}: \Gamma(\Bb_{\R})\rightarrow \Gamma(\Lambda^2\Bb_{\R}),\quad fX \mapsto f\widehat{\delta} X,\qquad f\in C^{\infty}(\mathscr{O}), X\in \{\hat{J}_i\}.\] We further define $\widehat{\delta}(f) = 0$ for $f\in  C^{\infty}(\mathscr{O})$, and extend $\widehat{\delta}$ to a degree one derivation of the exterior algebra $\Gamma(\Lambda \Bb_{\R})$.

We want to show that $\widehat{\delta}$ turns $\Bb_{\R}$ into a Lie bialgebroid. By \cite[Definition 3.12]{L-GSX11}, our goal is to check that $\widehat{\delta}^2=0$ and \begin{eqnarray} \label{Eqdelta1}\widehat{\delta}(fg) &=& g\widehat{\delta}(f)+f\widehat{\delta}(g),\\ \label{Eqdelta2}\widehat{\delta}(fX) &=& \widehat{\delta}(f)\wedge X+f\widehat{\delta} X,\\ \label{Eqdelta3} \widehat{\delta}(\lbrack X,Y\rbrack) &=& \lbrack \widehat{\delta} X,Y\rbrack +\lbrack X,\widehat{\delta} Y\rbrack.\end{eqnarray}  Now it is easy to see that it is sufficient to check these relations for $X,Y$ in the generating set $\{\hat{J}_i\}$ (and arbitrary functions), as long as we also check the identity \begin{eqnarray}\label{Eqdelta4}\widehat{\delta}(\lbrack X,g\rbrack) &=& \lbrack \widehat{\delta} X,g\rbrack +\lbrack X,\widehat{\delta} g\rbrack\end{eqnarray} on this set. 

The identities \eqref{Eqdelta1} and \eqref{Eqdelta2} hold by definition. For \eqref{Eqdelta4}, we have by definition that $\lbrack X,g\rbrack = \widehat{\rho}(X)(g)$, and as $\hat{\delta}$ is zero on functions, we have to prove $\lbrack \widehat{\delta} X,g\rbrack=0$ for $X$ in the generating set. This follows immediately from the fact that the elements of \eqref{Eqdelta0} do not contain $\hat{J}_0$ (and the anchor map is zero on $\{\hat{J}_1,\hat{J}_2,\hat{J}_3\}$).

It remains to verify $\widehat{\delta}^2=0$ and \eqref{Eqdelta3}. For the latter, we observe that, as $\widehat{\delta}$ is $C^{\infty}(\mathscr{O})$-linear, the cocycle relation for brackets between $\{\hat{J}_1,\hat{J}_2\}$ and $\{\hat{J}_1,\hat{J}_3\}$ follow from those for the bialgebras  $\bb$. Finally, as $\{\hat{J}_1,\hat{J}_2,\hat{J}_3\}$ commute with $\hat{J}_0$, it remains to verify that the right hand side of \eqref{Eqdelta3} is zero for $X\in \{\hat{J}_1,\hat{J}_2,\hat{J}_3\}$ and $Y=\hat{J}_0$. This boils down to proving \[\lbrack \hat{J}_{2},\hat{J}_{2}\wedge \hat{J}_3\rbrack =0,\quad \lbrack \hat{J}_{3},\hat{J}_{2}\wedge \hat{J}_3\rbrack =0, \quad \lbrack Y \hat{J}_2\wedge\hat{J}_3,\hat{J}_0\rbrack +\lbrack \hat{J}_1,\hat{J}_2\wedge \hat{J}_3\rbrack =0.\] The first identities are immediate since $\hat{J}_2$ and $\hat{J}_3$ commute in $\bb$. For the third identity, we compute \begin{eqnarray*} && \hspace{-2cm}   \lbrack Y\hat{J}_2\wedge \hat{J}_3,\hat{J}_0\rbrack +\lbrack \hat{J}_1,\hat{J}_2\wedge \hat{J}_3\rbrack  \\&=& \lbrack Y\hat{J}_2,\hat{J}_0\rbrack \wedge \hat{J}_3 + \lbrack \hat{J}_1,\hat{J}_2\rbrack \wedge \hat{J}_3 + \hat{J}_2\wedge \lbrack \hat{J}_1,\hat{J}_3\rbrack \\ &=& -\widehat{\rho}(\hat{J}_0)(Y)\hat{J}_2\wedge \hat{J}_3 +2\hat{J}_2\wedge \hat{J}_3 \\ &=& -2\hat{J}_2\wedge \hat{J}_3 +2\hat{J}_2\wedge \hat{J}_3 \\ &=& 0.
\end{eqnarray*}

Finally, to verify $\widehat{\delta}^2=0$, we can do a direct computation, but it suffices also to establish that the dual of $\widehat{\delta}$ leads, by the duality given before the statement of the proposition, to the Lie algebroid structure on $\mathfrak{A}_{\R}$. But since $\widehat{\delta} =0$ on functions, we indeed have that the dual anchor map is zero, and the duality formula \eqref{Eqdeltadual} is simply $\widehat{\delta}(X)(\xi,\eta) = X(\lbrack \xi,\eta\rbrack)$. The duality is then immediate from the formulae \eqref{Eqdelta0}.

\end{proof} 

\begin{Rem}\label{RemFormdelta} From the above duality, it follows that the cobracket $\delta$ of $\mathfrak{A}_{\R}$ is given by \[\delta(J_1) = 0 ,\quad \delta(J_2) = J_1\wedge J_2,\quad \delta(J_3)= J_1\wedge J_3,\quad \delta(E) =0,\quad \delta(Y) = 2J_0 .\] 
\end{Rem}

\begin{Rem} The above is in a sense \emph{dual} to the case of \emph{dynamical $\su(2)$} \cite{EV98} (see also \cite[Example 3.16.(4)]{L-GSX11}), where a Lie bialgebroid structure is constructed on the direct sum bundle $\su(2)\oplus T\R \rightarrow \R$. However, an important difference in this case is that the cobracket does not vanish on functions, giving the Lie bialgebroid its `dynamical' character. 

\end{Rem}

\section{A contraction Poisson groupoid}

Our next goal is to integrate the Lie bialgebroids of the previous section.

\begin{Def}\label{DefLieGrpd} A \emph{Lie groupoid} consists of a groupoid whose arrow space $\mathscr{G}$ and object space $\mathscr{O}$ are both smooth manifolds, whose source and target maps $\mathscr{G} \overset{s}{\underset{t}{\rightrightarrows}}  \mathscr{O}$ are smooth submersions, and all of whose structure maps are smooth. 
\end{Def} 

Note that \[\mathscr{G}^{(2)} := \{(f,g) \in \mathscr{G}\times \mathscr{G} \mid t(f) = s(g)\} \subseteq \mathscr{G}\times \mathscr{G}\] is a smooth submanifold by the assumptions on $s$ and $t$. Hence it makes sense to ask that the multiplication map $\mathscr{G}^{(2)}\rightarrow \mathscr{G}$ is smooth.  A Lie groupoid with $\mathscr{O}$ a single point is precisely a Lie group. A subtle point is that one does not always require $\mathscr{G}$ to be Hausdorff -- this can already happen when $\mathscr{G}$ is a bundle of Lie groups, as we will see later.

Any Lie groupoid gives rise to a Lie algebroid by the correspondence
\[ (\mathscr{O},\mathscr{G},s,t)\qquad  \Rightarrow \qquad   A = \ker(\rd s)_{|\mathscr{O}}, \quad \rho = \rd t_{\mid A}.\] The converse direction (Lie's third theorem) is more complicated in the Lie algebroid setting and not always possible in general \cite{CF03}.

Our Lie algebroid $\Bb_{\R}$, being a direct sum of two simple Lie algebroids, is trivially integrable, with associated Lie groupoid \[B_{\R} = \R\times B\times \R,\] where $B$ is the Lie group \[B = \{\begin{pmatrix} a & 0 \\ b & a^{-1}\end{pmatrix}\mid a\in \R_+,b\in \C\},\] and where the product is given by \[(x,g,y)(y,h,z) = (x,gh,z).\] Here the real Lie algebra $\bb$ is identified as the Lie algebra of $B$ by the matrix representation \eqref{EqMatDual}.

We will need to impose a compatible Poisson structure on our Lie groupoid. The following notion of \emph{Poisson groupoid} was introduced in \cite{Wei88}, see also \cite[Definition 3.1]{L-GSX11}.

\begin{Def} A Lie groupoid $\mathscr{G}$ is called a \emph{Poisson groupoid} if $\mathscr{G}$ is equipped with a Poisson structure for which the graph \[\graph(\mult) = \{(f,g,h)\mid (f,g)\in \mathscr{G}^{(2)}, h \in \mathscr{G}, fg=h\}\] of the multiplication map $\mult: \mathscr{G}^{(2)} \rightarrow \mathscr{G}$ is a coisotropic submanifold of $\mathscr{G}\times \mathscr{G}\times \overline{\mathscr{G}}$.
\end{Def} 

Here $\overline{\mathscr{G}}$ means $\mathscr{G}$ with the opposite Poisson structure. Recall further that a submanifold $P\subseteq M$ of a Poisson manifold is called \emph{coisotropic} if the Hamiltonion of a function which disappears on $P$ is tangent to $P$. Equivalently, the pairing between the conormal bundle $N^*P$ and $TM_{\mid P}$, induced by the Poisson bivector, factors through $TP$.

A Poisson groupoid with $\mathscr{O}$ a single point is a Poisson-Lie group.

The infinitesimal structure associated to a Poisson groupoid will be precisely a Lie bialgebroid structure on the associated Lie algebroid $A$ over $\mathscr{O}$. Namely, by means of the pullback with $t$ we can identify $C^{\infty}(\mathscr{O})$ with the subalgebra $t^*C^{\infty}(\mathscr{O})\subseteq C^{\infty}(\mathscr{G})$ of left invariant functions, and we can identify $\Gamma(A)$ with left invariant tangent vector fields $\Gamma_l(T\mathscr{G})$. Then the operation $\delta$ determined by the equations \[\delta(f)(dg)= 2\{g,f\}, \quad \delta(X)(df,dg) = X(\{f,g\})-\{Xf,g\}-\{f,Xg\}\] can be shown to map $t^*C^{\infty}(\mathscr{O})$ and $\Gamma_l(T\mathscr{G})$ into respectively $\Gamma_l(T\mathscr{G})$ and $\Gamma_l(\Lambda^2T\mathscr{G})$, and to give in effect a Lie bialgebroid, see \cite[Corollary 3.10]{L-GSX11} and the discussion following it.

Denote now by $\Pi$ the fundamental bivector associated to a Poisson manifold. In the following theorem, we introduce a Poisson groupoid structure on $B_{\R}$. The formulas will use the identification $T_bB \subseteq T_bM_2(\C) \cong M_2(\C)$.

\begin{Theorem} Define on $B_{\R} = \R\times B \times \R$ the bivector $\Pi$ by \[\Pi(x,g,y) = \Pi_{x,y}(g),\] where $\Pi_{x,y}$ is the following bivector on $B$: \[\Pi_{x,y}(\begin{pmatrix} a & 0 \\ z & a^{-1}\end{pmatrix}) = \begin{pmatrix} 0 & 0 \\ iz & 0\end{pmatrix} \wedge \begin{pmatrix} -a & 0\\ 0 & a^{-1}\end{pmatrix} + \frac{1}{4}(xa^2-ya^{-2}) \begin{pmatrix} 0 & 0 \\ i & 0 \end{pmatrix} \wedge \begin{pmatrix} 0 & 0 \\ 1& 0\end{pmatrix}.\] Then $\Pi$ is a Poisson bivector turning $B_{\R}$ into a Poisson groupoid with $\Bb_{\R}$ as its associated Lie bialgebroid. 
\end{Theorem} 

\begin{proof} First of all, we have to verify that $\Pi$ is a Poisson bivector. This is equivalent with showing that each $\Pi_{x,y}$ is a Poisson bivector (since $\Pi$ does not contain any derivations in the $x$- or $y$-directions). It is however easy to see that upon expanding the Schouten-Nijenhuis bracket $\lbrack \Pi_{x,y},\Pi_{x,y}\rbrack$, all terms will vanish for trivial reasons. Hence $\Pi$ is a Poisson vector. 

To verify that $\Pi$ turns $B_{\R}$ into a Poisson groupoid, we check the conditions in \cite[Theorem 3.8]{L-GSX11}. In fact, since $\Pi$ vanishes on $\id_{\R} = \{(x,e_B,x)\mid x\in \R\} \subseteq B_{\R}$, the first `bitranslation condition' becomes the identity \[\Pi_{x,z}(gh) = \Pi_{x,y}(g)h + g\Pi_{y,z}(h),\qquad x,y,z\in\R,g,h\in B,\] which is easily checked by a direct computation. The fact that $\id_{\R}$ is coisotropic follows immediately from the fact that $\Pi$ does not contain any derivations in the $x$- or $y$-directions. For this same reason, the other conditions are trivially satisfied. 

It remains to check that $\Bb_{\R}$ is the Lie bialgebroid of $B_{\R}$. Let $\widehat{\delta}$ be the cobracket associated to $\Pi$. Then $\widehat{\delta}(f)=0$ for $f\in C^{\infty}(\R)$ follows immediately. Let us check that \[\widehat{\delta}(\hat{J}_0) = 2\hat{J}_2\wedge \hat{J}_3.\] In fact, the left translated vector field associated to $\hat{J}_0$ is given by $(\hat{J}_0)_{(x,g,y)} = 2\frac{\partial}{\partial y}$. An easy computation then shows that, for $g= \begin{pmatrix} a & 0 \\ z & a^{-1}\end{pmatrix}$, \[\lbrack \hat{J}_0, \Pi\rbrack_{(x,g,y)} = 2a^{-2} \hat{J}_2 \wedge \hat{J}_3,\] and so indeed the formula for $\widehat{\delta}(\hat{J}_0)$ holds. The other values of $\widehat{\delta}$ are also easily computed. 
\end{proof} 

As an immediate corollary, we obtain:

\begin{Cor} Let $A,B,C$ be the coordinate functions of the map \[(x,\begin{pmatrix} a & 0 \\ z & a^{-1}\end{pmatrix},y) \mapsto (a,z,\overline{z}).\] Then the Poisson bracket satisfies  \[\{A,B\} = iAB,\quad \{A,C\} = -iAC,\quad \{B,C\} = \frac{i}{2}(x A^2- y A^{-2}),\] while the bracket with any $f\in C^{\infty}(\R\times\R)$ is zero. 
\end{Cor}

\begin{Rem} The $(B,\Pi_{x,y})$ can be seen as a right \emph{Poisson torsor} for the Poisson group $(B,\Pi_{y,y})$, in that $B$ is a right $B$-torsor for an action which is Poisson with respect to the stated Poisson structures. 
\end{Rem} 

Let us now turn towards the integration of the dual Lie bialgebroid $\mathfrak{A}_{\R}$. This is a trickier question. By general theory \cite{DL66}, one knows that the bundle of Lie algebras $\mathfrak{A}_{\R}$ integrates to a bundle of simply connected Lie groups, but this bundle will not be Hausdorff, as the reasoning in \cite[Section IV.3]{DL66} shows. Moreover, the fibers will not necessarily be matrix groups. Therefore, rather than describe a full bundle of Lie groups, we will describe a bundle of \emph{local Lie groups} integrating $\mathfrak{A}_{\R}$. 

In the following, we denote by $\Arg$ the $(-\pi,\pi)$-valued argument function on $\C\setminus \R^-$. For a matrix $m \in M_2(\C)$, we write \[m = \begin{pmatrix} a(m) & b(m) \\ c(m) & d(m)\end{pmatrix}.\]

\begin{Prop}\label{PropBundle} Fix $y\in \R$, and define $G_y \subseteq SL(2,\C)$ as the subspace \[G_y = \{\begin{pmatrix} a & b \\ -y\bar{b} & \bar{a} \end{pmatrix} \mid |a|^2+y|b|^2 = 1,  - \frac{\pi}{3}<\Arg(a)<\frac{\pi}{3}\}.\] Define on $G_y^{(2)} = \{(g,h)\in G_y\times G_y \mid gh \in G_y\}$ the function \[\Omega_y(g,h) = \frac{1}{y}\Arg\left(\frac{a(gh)}{a(g)a(h)}\right),\] where \[\Omega_0(g,h) = \Imm\left(-\frac{b(g)\overline{b(h)}}{a(g)a(h)} \right) = \lim_{y\rightarrow 0} \Omega_y(g,h).\] Then \[\widetilde{G}_y = \{(r,g)\mid r\in \R,g\in G_y\}\] becomes a local Lie group for the multiplication \[(r,g)\cdot (s,h) = (r+s+\Omega_y(g,h),gh).\] Moreover, the Lie algebra of $\widetilde{G}_y$ is isomorphic to $\su_{ext}^{(y)}(2)$.
\end{Prop} 

\begin{proof} We leave it as an easy exercise to the reader to verify  that $\widetilde{G}_y$ is a local Lie group. We only check here that it has the correct Lie algebra. In the following, the case $y=0$ is always seen as a limit $y\rightarrow 0$.

Consider the following one-parametergroups, where $y^{1/2} = i|y|^{1/2}$ for $y<0$, \begin{eqnarray*}\phi_0^{(y)}(t) &=& (t,I_2),\\ \phi_1^{(y)}(t) &=&  (0,\begin{pmatrix} e^{it} & 0 \\ 0 & e^{-it}\end{pmatrix}),\\ \phi_2^{(y)}(t) &=& (0, \begin{pmatrix} \cos(y^{1/2}t) & y^{-1/2}\sin(y^{1/2}t) \\ -y^{1/2}\sin(y^{1/2}t) & \cos(y^{1/2}t)\end{pmatrix})\\
\phi_3^{(y)}(t) &=& (0, \begin{pmatrix} \cos(y^{1/2}t) & iy^{-1/2}\sin(y^{1/2}t) \\ iy^{1/2}\sin(y^{1/2}t) & \cos(y^{1/2}t)\end{pmatrix}).\\
\end{eqnarray*} Then the associated tangent vectors are 
\[ J_0^{(y)} = (1,0),\quad J_1^{(y)} =(0, \begin{pmatrix} i & 0 \\0 & -i\end{pmatrix}),\quad J_2^{(y)} = (0,\begin{pmatrix} 0 & 1 \\ -y & 0\end{pmatrix}),\quad  J_3^{(y)} = (0,\begin{pmatrix} 0 & i\\ iy& 0 \end{pmatrix}).\] It is clear that $J_0^{(y)}$ is central, and that \[\lbrack J_1^{(y)},J_2^{(y)}\rbrack = 2J_3^{(y)},\quad \lbrack J_3^{(y)},J_1^{(y)} \rbrack = 2J_2^{(y)}.\] Finally, we compute 
\begin{eqnarray*} \lbrack J_2^{(y)},J_3^{(y)}\rbrack &=& \lim_{s\rightarrow 0} \frac{1}{s}(\phi_2^{(y)}(-\sqrt{s})\phi_3^{(y)}(-\sqrt{s})\phi_2^{(y)}(\sqrt{s})\phi_3^{(y)}(\sqrt{s}) - (0,I_2)) \\ 
&=& 2yJ_1^{(y)}+2J_0^{(y)}.
\end{eqnarray*}
 
\end{proof} 

\begin{Cor} Consider \[A_{\R} = \{(y,g)\mid g\in \widetilde{G}_y\}\subseteq \R\times SL(2,\C).\] Then $A_{\R}$ is a bundle of local Lie groups with associated Lie algebroid $\mathfrak{A}_{\R}$. 
\end{Cor} 
\begin{proof} The only thing which has to be observed is that the applications $(y,t) \rightarrow \phi_i^{(y)}(t)$ in the previous proof are smooth. Hence infinitesimal left multiplication with respect to these one-parameter-fields give left-invariant smooth vector fields on $A_{\R}$, and the same argument as in the previous proof shows that their restrictions to $\R$ give the Lie algebroid $\mathfrak{A}_{\R}$.

\end{proof}

Finally, we give a formula for the Poisson bivector on $A_{\R}$. Again, we will only succeed in giving a nice expression locally in a neighbourhood of the point $(0,e_{\widetilde{G}_0})$, the reason being that there does not seem to be a canonical global \emph{transversal} vector field that one can use. More precisely, let us restrict our object manifold $\R$ to an open interval of the form $(-M,M)$ for some $M>0$, and consider the restricted bundle $A_{(-M,M)}$. Then we can find an open neighborhood $U_{(-M,M)}$ around $\id_{(-M,M)}$ such that the following holds: for each $y\in (-M,M)$, the map \[\lambda_y: U_0\rightarrow U_y, \quad g=g_0\mapsto g_y\] is a diffeomorphism, where, using the notation from the proof of Proposition \ref{PropBundle},
\[\lambda_y\left(\phi_0^{(0)}(r)\phi_1^{(0)}(s)\phi_2^{(0)}(u)\phi_3^{(0)}(v)\right) = \phi_0^{(y)}(r)\phi_1^{(y)}(s)\phi_2^{(y)}(u)\phi_3^{(y)}(v).\] It then follows that $y\mapsto g_y$ is a section of $U_{(-M,M)}$ for each $g\in U^{(0)}\subseteq \widetilde{G}^{(0)}$, and we can define the transversal vector field $\frac{\partial}{\partial y}$ in $(y,g_y)$ as the tangent vector at zero to the path $t \mapsto (y+t,g_{y+t})$. 

\begin{Theorem} The bialgebroid structure on $\mathfrak{A}_{\R}$ integrates to a (local) Poisson groupoid structure on the bundle $U_{(-M,M)}$ by means of the Poisson bivector \[\Pi(y,r,\begin{pmatrix} a & b\\ -y\bar{b} & \bar{a}\end{pmatrix}) = 2J_0^{(y)}\wedge \frac{\partial}{\partial y}   + J_1^{(y)}\wedge (\Ree(ab)J_2^{(y)}+\Imm(ab) J_3^{(y)}) +|b|^2 J_2^{(y)}\wedge J_3^{(y)}.\]
\end{Theorem} 

Here we have identified $\su_{\ext}^{(y)}(2)$ as left invariant vector fields on $\widetilde{G}_y$. We also recall that $J_0^{(y)} = \frac{\partial}{\partial r}$. 

\begin{proof} By general theory \cite{MX00}, we know that the bialgebroid structure on $\mathfrak{A}_{(-M,M)}$ will integrate to a Poisson structure on $U_{(-M,M)}$. Let $p: U_{(-M,M)}\rightarrow \R$ be the projection map. Then, from the discussion in \cite[Section 3.2]{L-GSX11}, we know that $\lbrack p^*f,\Pi\rbrack$ has to be left invariant and equal to $\delta(f)$ for $f\in C^{\infty}(-M,M)$. We can hence conclude from Remark \ref{RemFormdelta} that \[\Pi(y,g) = 2\frac{\partial}{\partial r}\wedge \frac{\partial}{\partial y}   + \widetilde{\Pi}(y,g),\] with $\widetilde{\Pi}(y,g)\in \Lambda^2T_{g}(\widetilde{G}_y)$. In particular, since $\id_{(-M,M)} \subseteq U_{(-M,M)}$ via $\id: y \mapsto (y,e_{\widetilde{G}_y})$ is coisotropic, we conclude that $\Pi(y,e_{\widetilde{G}_y}) =   2\frac{\partial}{\partial r}\wedge \frac{\partial}{\partial y}$.

Now from the multiplicative formula \cite[Theorem 3.8.(1)]{L-GSX11} for $\Pi$, we get that for $\mathbf{g},\mathbf{h}$ the canonical sections $y\mapsto g_y$ and $y\mapsto h_y$ for $g,h\in U_{(-M,M)}$ with $gh\in U_{(-M,M)}$, we have  \[\Pi(g_yh_y) = \mathbf{g}\Pi(h_y) + \Pi(g_y)\mathbf{h} - \mathbf{g}\Pi(y,e_{\widetilde{G}_y})\mathbf{h}.\] 

Consider now again the one-parametergroups $\phi^{(y)}_i(t)$ appearing in the proof of Proposition \ref{PropBundle}. By definition we have that $\phi_i^{(y)}(t) = \phi_i(t)_y$, where $\phi_i(t) = \phi^{(0)}_i(t)$. This implies that \[\left(\frac{\partial}{\partial y}\right)_{\mid (y,\phi_i^{(y)}(s))} \cdot \boldsymbol{\phi}_i(t) = \boldsymbol{\phi}_i(t) \cdot \left(\frac{\partial}{\partial y}\right)_{\mid (y,\phi_i^{(y)}(s))}  =  \left(\frac{\partial}{\partial y}\right)_{\mid (y,\phi_i^{(y)}(s+t))}.\]

For these particular values, our multiplicative formula becomes \[\widetilde{\Pi}(\phi_i^{(y)}(t+s)) = \widetilde{\Pi}(\phi_i^{(y)})(t) \phi_i^{(y)}(s) + \phi_i^{(y)}(t)\widetilde{\Pi}(\phi_i^{(y)}(s)).\] Hence we can apply the general method and formulas as in \cite[Section 2.3.1]{ES02} to conclude that \begin{eqnarray*} \widetilde{\Pi}(y,\phi_1^{(y)}(t)) &=& 0, \\ \widetilde{\Pi}(y,\phi_2^{(y)}(t))  &=& y^{-1/2}\sin(y^{1/2}t) \cos(y^{1/2}t) J_1^{(y)}\wedge J_2^{(y)} + (y^{-1/2}\sin(y^{1/2}t))^2 J_2^{(y)} \wedge J_3^{(y)}  \\  \widetilde{\Pi}(y,\phi_3^{(y)}(t))  &=& y^{-1/2}\sin(y^{1/2}t) \cos(y^{1/2}t) J_1^{(y)}\wedge J_3^{(y)} + (y^{-1/2}\sin(y^{1/2}t))^2 J_2^{(y)} \wedge J_3^{(y)} \\  \widetilde{\Pi}(y,\phi_4^{(y)}(t))  &=&  0.\end{eqnarray*}

As one has \[\phi_0^{(y)}(t)\phi_1^{(y)}(s)\phi_2^{(y)}(u)\phi_3^{(y)}(v)= \left(\phi_0(t)\phi_1(s)\phi_2(u)\phi_3(v)\right)_y\] by definition as well, we obtain that one can apply the ordinary multiplicative formula for $\widetilde{\Pi}$ on elements of the form $\phi_0^{(y)}(t)\phi_1^{(y)}(s)\phi_2^{(y)}(u) \phi_3^{(y)}(v)$. An easy computation then reveals that $\Pi$ has the form as in the statement of the theorem. 
\end{proof} 

\begin{Cor} Write \[(y,r,a,b,c,d):U_{(-M,M)} \rightarrow \C^6,\quad \left(y,(r,\begin{pmatrix} a & b \\ -y\bar{b} & \bar{a}\end{pmatrix})\right) \mapsto (y,r,a,b,-\bar{b},\bar{a}).\] Then we have the following Poisson brackets: \[ \{a,b\} = iab, \quad \{a,c\} = iac,\quad  \{a,d\} = 2iybc,\quad \{b,c\} = 0, \quad \{b,d\} = ibd,\quad \{c,d\} = icd,  \]
 \[\{r,-\} = 2\frac{\partial}{\partial y},\qquad \{y,-\} = -2\frac{\partial}{\partial r}.\]
\end{Cor}

 Note however that for example the bracket $\{r,a\} = 2 \frac{\partial a}{\partial y}$ does not have a particularly nice expression.

\bibliographystyle{habbrv}
\bibliography{RefPoisGroup}

\end{document}